\documentclass[12pt,a4paper]{article}
\usepackage[left=1in,right=1in,top=1in,bottom=1in,a4paper]{geometry}
\usepackage{amsfonts,amsmath,amsthm}
\usepackage{graphicx,dsfont}
\usepackage{pst-plot,pstricks,multido}
\usepackage{enumerate,enumitem}
\usepackage{amssymb}
\usepackage{hyperref}

\newtheorem{thm}{Theorem}[section]
\newtheorem{lem}[thm]{Lemma}

\theoremstyle{definition}

\newcommand{\kc}{\mathcal{C}}

\makeatletter
\newcommand{\specificthanks}[1]{\@fnsymbol{#1}}
\makeatother

\title{On all Pickands Dependence
Functions whose corresponding Extreme-Value-Copulas have \mbox{Spearman $\rho$} (Kendall $\tau$) 
identical to some value $v \in [0,1]$}

\author{Noppadon Kamnitui\thanks{Department for Mathematics, University of Salzburg, Salzburg, Austria, E-mails: noppadon.kam@gmail.com, wolfgang@trutschnig.net}\label{uni-salzburg} \and 
Christian Genest\thanks{Department of Mathematics and Statistics, McGill University, 805, rue Sherbrooke ouest, Montr\'eal (Qu\'ebec), Canada H3A 0B9, E-mail: christian.genest@mcgill.ca} \and
Wolfgang Trutschnig\textsuperscript{\specificthanks{1}}}

\begin{document}

\maketitle
\begin{abstract}
We answer an open question posed by the second author at the Salzburg workshop on Dependence Models and Copulas in 2016 concerning the size of the family $\mathcal{A}^\rho_v$ ($\mathcal{A}^\tau_v$) of all Pickands dependence functions $A$ whose
corresponding Extreme-Value-Copulas have Spearman $\rho$ (Kendall $\tau$) equal to some arbitrary, fixed value $v \in [0,1]$. After determining compact sets 
$\Omega^\rho_v, \Omega^\tau_v \subseteq [0,1] \times [\frac{1}{2},1]$ containing the graphs of all 
Pickands dependence functions from the classes $\mathcal{A}^\rho_v$ and $\mathcal{A}^\tau_v$ respectively, we then show that both sets are best possible.        
\end{abstract}

\section{Notation and preliminaries}
Let $\kc$ denote the family of all two-dimensional copulas, and $\kc_{ex}$ the class of all extreme-value copulas. As usual, $d_\infty$ denotes the uniform metric on $\kc$.
$\mathcal{A}$ will denote the set of all Pickands dependence functions (see \cite{TSFS}). Arzela-Ascoli thereom (\cite{Ru}) implies that 
$\mathcal{A}$ is a compact (and convex) subset of the Banach space $C([0,1],\Vert \cdot \Vert_\infty)$.
For every $A \in \mathcal{A}$ the corresponding EVC will be denoted by $C_A$, i.e. 
$$
C_A(x,y)=(xy)^{A\big(\frac{\ln{x}}{\ln{xy}}\big)}.
$$
It is well-known that Spearman's $\rho$ and Kendall's $\tau$ can be calculated as  
\begin{align}
\rho(C_A)&=-3+12\int_{[0,1]}\frac{1}{(1+A(t))^2}d \lambda(t), \\
\tau(C_A)&=1-\int_{0}^{1}\left(1+(1-t)\frac{A'(t)}{A(t)}\right)\left(1-t\frac{A'(t)}{A(t)}\right)d \lambda(t) = \int_{[0,1]} \frac{t(1-t)}{A(t)} dA'(t).
\end{align}
For every $v \in [0,1]$ we will let $\mathcal{A}^\rho_v$ ($\mathcal{A}^\tau_v$) denote the family of all Pickands dependence functions $A$ fulfilling $\rho(C_A)=v$ ($\tau(C_A)=v$). 
Given $A,B \in \mathcal{A}$ we will write $A \succ B$ if $A(x) \geq B(x)$ for all $x \in [0,1]$ with strict inequality in at least one point (and hence on an interval).

In the sequel we will work with the following five types of piecewise linear Pickands dependence functions, see Figure \ref{fig:aux}. \\
(i) Consider $y \in [\frac{1}{2},1]$ and $x \in [1-y,y]$ and let $T_{x,y} \in \mathcal{A}$ be defined by 
\begin{equation*}\label{defTxb}
T_{x,y}(t)=\begin{cases}
-\frac{1-y}{x}\,t+1 & \text{if } t\in[0,x], \\
-\frac{1-y}{1-x}\,(1-t)+1 & \text{if } t\in(x,1]. \\
\end{cases}
\end{equation*}
(ii) For every $y \in [\frac{1}{2},1]$ define $L_y \in \mathcal{A}^s$ by
\begin{equation*}\label{defLy}
L_y(t)=\begin{cases}
1-t & \text{if } t \in [0,1-y], \\
y & \text{if } t \in (1-y,y], \\
t &  \text{if } t \in (y,1]. \\
\end{cases}
\end{equation*}
(iii) For every $x\in[0,\frac{1}{2}]$ and $y\in[\frac{1}{2},1]$ define $P_{x,y}\in\mathcal{A}$ by
\begin{equation*}
P_{x,y}(t)=
\begin{cases}
1-t & \text{if } t \in [0,x], \\
\frac{1-x-y}{x-y}(t-y)+y & \text{if } t \in (x,y], \\
t & \text{if } t \in (y,1].
\end{cases}
\end{equation*}
(iv) For every $x\in(0,\frac{1}{2}]$, $y\in[1-x,1]$ and let $Z_{x,y}\in\mathcal{A}^s$ be defined by
\begin{equation*}
Z_{x,y}(t)=\begin{cases}
-\frac{1-y}{x}t+1 & \text{if }t\in[0,x], \\
y & \text{if }t\in(x,1-x], \\
-\frac{1-y}{x}(1-t)+1 & \text{if }t\in(1-x,1]. \\
\end{cases}
\end{equation*}
(v) For every $x\in[0,\frac{1}{2})$ and $y\in[\frac{1}{2},1-x]$ define $W_{x,y}\in\mathcal{A}^s$ by
\begin{equation*}
W_{x,y}(t)=\begin{cases}
1-t & \text{if }t \in [0,x], \\
\frac{y-1+x}{\frac{1}{2}-x}(t-\frac{1}{2})+y & \text{if }t \in (x,\frac{1}{2}], \\
-\frac{y-1+x}{\frac{1}{2}-x}(t-\frac{1}{2})+y & \text{if }t \in (\frac{1}{2},1-x], \\
t & \text{if }t \in (1-x,1]. \\
\end{cases}
\end{equation*}

\begin{figure}[h!]
	\begin{center}
		\includegraphics[width = 14cm]{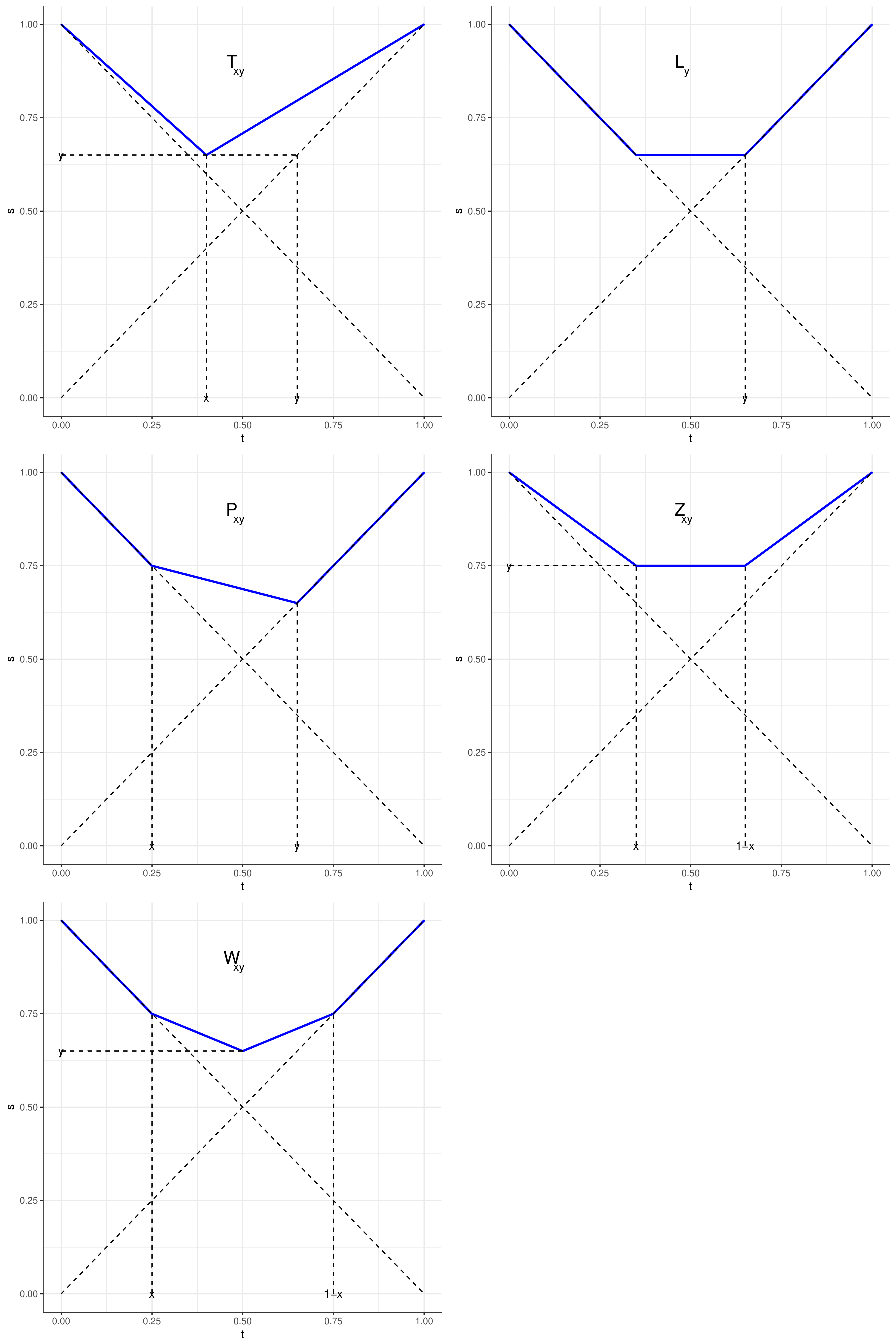}
		\caption{The five types of auxiliary functions considered.}\label{fig:aux}
	\end{center}
\end{figure} 

\clearpage
\section{Spearman's $\rho$}\label{sec:rho.full}

A straightforward calculation yields $\int_{[0,1]}\frac{1}{(1+T_{x,y}(t))^2}d \lambda(t) = \frac{1}{2(1+y)}$. Hence, defining $\varphi_1:[\frac{1}{2},1] \rightarrow [0,1]$ by
$$
\varphi_1(y)=-3 + \frac{6}{1+y}
$$
we get
\begin{equation}\label{eqTxb}
\rho(C_{T_{x,y}}) = \varphi_1(y)
\end{equation}
for every $T_{x,y}$ with $x \in [1-y,y]$, i.e. $\rho(C_{T_{x,y}})$ does not depend on $x \in [1-y,y]$. Obviously $\varphi_1$ is a strictly decreasing homeomorphism mapping $[\frac{1}{2},1]$ onto 
$[0,1]$. Furthermore, letting $\varphi_1^{-1}(\rho)=\frac{3-\rho}{3+\rho}$ denote its inverse,  
\begin{equation}\label{eqvarphi1}
T_{x,\varphi_1^{-1}(\rho_0)} \in \mathcal{A}_{\rho_0}
\end{equation}
holds for every $x \in [1-\varphi_1^{-1}(\rho_0),\varphi_1^{-1}(\rho_0)]$. 

A straightforward calculation yields $\int_{[0,1]}\frac{1}{(1+P_{x,y}(t))^2}d \lambda(t) = 
\frac{1-x+xy}{(2-x)(1+y)}$, implying
\begin{equation}\label{eq:Pxy}
\rho(C_{P_{x,y}})= - 3 + 12 \frac{1-x+xy}{(2-x)(1+y)}.
\end{equation}
Considering $L_y=P_{1-y,y}$ we get  
$$
\int_{[0,1]}\frac{1}{(1+L_{y}(t))^2}d \lambda(t) = \frac{y(2-y)}{(1+y)^2}.
$$
and defining $\psi_1:[\frac{1}{2},1] \rightarrow [0,1]$ by
$$
\psi_1(y)=-3 + 12\,\frac{y(2-y)}{(1+y)^2}
$$
we get
\begin{equation}\label{eqLb}
\rho(C_{L_y}) = \psi_1(y)
\end{equation}
for every $y \in [\frac{1}{2},1]$. Notice that for $y_1<y_2$ obviously $L_{y_2} \succ L_{y_1}$, hence $\rho(C_{L_{y_1}})> \rho(C_{L_{y_2}})$, so $\psi_1$ is a strictly decreasing homeomorphism mapping 
$[\frac{1}{2},1]$ onto $[0,1]$.  
Letting $\psi_1^{-1}$ denote its inverse,
\begin{equation}\label{eqpsi1}
L_{\psi_1^{-1}(\rho_0)} \in \mathcal{A}^\rho_{\rho_0}
\end{equation}
holds for every $\rho_0 \in [0,1]$.  

For every $y \in [\frac{1}{2},1]$ and every $x \in [0,\frac{1}{2}]$ obviously 
$\rho(C_{P_{x,y}}) \in [\rho(C_{T_{y,y}}),1]=[\varphi_1(y),1]$. Additionally, for 
$\rho_0 \in [0,1)$, fixed $y \in (\frac{1}{2},1]$ and $0 \leq x_1 < x_2 \leq 1$ we have 
$P_{x_1,y} \succ P_{x_2,y}$, implying that the mapping $x \mapsto \rho(C_{P_{x,y}})$ is an increasing homeomorphism mapping 
$[0,\frac{1}{2}]$ onto $[\varphi_1(y),1]$. 

Again consider $\rho_0 \in [0,1)$. Then for every $y \geq \varphi_1^{-1}(\rho_0) > \frac{1}{2}$ there exists exactly one $x \in [0,\frac{1}{2}]$ with $\rho(C_{P_{x,y}})=\rho_0$. It is easy to verify that this $x$ is given by 
$$
x=\frac{2(-3+\rho_0+3y+\rho_0y)}{-9+\rho_0+15y+\rho_0y}=:h_{\rho_{0}}(y). 
$$
and that $h_{\rho_{0}}(\varphi_1^{-1}(\rho_0))=0$ as well as 
$h_{\rho_{0}}(1)= \frac{2\rho_0}{3+\rho_0}= 1- \varphi_1^{-1}(\rho_0)$ holds.
Since the derivative $h_{\rho_{0}}'$ of $h_{\rho_{0}}$ is given by
$$
h'_{\rho_{0}}(y)=\frac{36(1-\rho_{0})}{(-9+\rho_{0}+15y+\rho_{0}y)^2} 
$$ 
we get $\min_{y\in[\varphi_1^{-1}(\rho_{0}),1]}h_{\rho_{0}}'(y)=h_{\rho_{0}}'(1)=\frac{36(1-\rho_{0})}{(6+2\rho_{0})^2} >0$. 
As direct consequence, $h_{\rho_{0}}$ is strictly increasing and for $\varphi_1^{-1}(\rho_0) \leq \underline{y} < \overline{y} \leq 1$ 
\begin{align}\label{eq:lip.invh1}
h_{\rho_{0}}(\overline{y})-h_{\rho_{0}}(\underline{y}) \geq h_{\rho_{0}}'(1)(\overline{y}-\underline{y}) = \frac{36(1-\rho_{0})}{(6 + 2\rho_{0})^2}\,(\overline{y}-\underline{y})
\end{align} 
This non-contractivity property will be key in the proof of Lemma \ref{upperboundrho}.

Straightforward (but tedious) calculations (see Appendix) show that the upper envelope $U_{\rho_0}$
of the family $(P_{h_{\rho_0}(y),y})_{y \in [\varphi_1^{-1}(\rho_0),1]}$ is given by 
\begin{align}\label{eq:U}
U_{\rho_0}(t)=
\begin{cases}
P_{0,\varphi_1^{-1}(\rho_{0})}(t) & \text{if } t \in \big[0,\frac{3-\rho_{0}}{6+\rho_{0}}\big), \\
\frac{9-\rho_0+4\sqrt{6-2\rho_0-15t(1-t)-\rho_0t(1-t)}}{15+\rho_0} & \text{if } t \in \big[\frac{3-\rho_{0}}{6+\rho_{0}},\frac{3+2\rho_{0}}{6+\rho_{0}}\big], \\
P_{1-\varphi_1^{-1}(\rho_{0}),1}(t) & \text{if } t \in \big(\frac{3+2\rho_{0}}{6+\rho_{0}},1\big],
\end{cases}
\end{align}
The function $U_{\rho_0}$ is symmetric w.r.t. $t=\frac{1}{2}$, convex and continuously differentiable on $(0,1)$.

We now show that $L_{\varphi_1^{-1}(\rho_0)} \leq A \leq U_{\rho_0}$ holds for every $A \in \mathcal{A}^\rho_{\rho_0}$ and $\rho_0 \in [0,1]$ and start with the lower bound.

\begin{lem}
Suppose that $\rho_0 \in [0,1]$. Then every $A \in \mathcal{A}^\rho_{\rho_0}$ fulfills $A \geq L_{\varphi_1^{-1}(\rho_0)}$.
\end{lem} 

\begin{proof}
The assertion is trivial for $\rho_0 \in \{0,1\}$, so we may consider $\rho_0 \in (0,1)$. Suppose that $A \in \mathcal{A}^\rho_{\rho_0}$ and 
that there exists some $t_0 \in (1-\varphi_1^{-1}(\rho_0),\varphi_1^{-1}(\rho_0))$ with $A(t_0) < L_{\varphi_1^{-1}(\rho_0)}(t_0)=\varphi_1^{-1}(\rho_0)$.  
Convexity of $A$ together with $A(0)=A(1)=1$ implies $A \leq T_{t_0,A(t_0)}$, from which we immediately get 
$$
\rho(C_A) \geq \rho(C_{T_{t_0,A(t_0)}}) = \varphi_1 \circ A(t_0) > \varphi_1 \circ \varphi_1^{-1}(\rho_0)=\rho_0,
$$ 
a contradiction to $A \in \mathcal{A}^\rho_{\rho_0}$.	
\end{proof}	

\begin{lem}\label{upperboundrho}
Suppose that $\rho_0 \in [0,1]$. Then every $A \in \mathcal{A}^\rho_{\rho_0}$ fulfills $A \leq U_{\rho_0}$.
\end{lem}
\begin{proof}
For the extreme cases $\rho_0=0$ and $\rho_0=1$ we get $U_{0}=1$ and $U_1=P_{0,\frac{1}{2}}$ respectively, so the result obviously holds. 
For the rest of the proof we consider $\rho_0 \in (0,1)$.  
Suppose that $A\in\mathcal{A}^\rho_{\rho_{0}}$ and that $s_0:=A(t_{0})>U_{\rho_{0}}(t_{0})$ for some $t_{0}\in[0,1]$. The case $s_0=1$ yields a contradiction immediately, so we assume $s_0 < 1$.   Symmetry of $U_{\rho_0}$ implies that it suffices to consider $t_0 \in (0,\frac{1}{2}]$, continuity of $U_{\rho_0}$ yields $r=\min_{t\in[0,1]}\sqrt{(t-t_0)^2+(U_{\rho_{0}}(t)-s_0)^2}>0$.\\
\noindent Define $f^*$ by $f^*(t)=\frac{s_0-1}{t_0-1}(t-t_0)+s_0$ (increasing green with positive slope in Figure \ref{fig:construction}). Convexity of $A$ 
yields $f^*(t)\leq A(t)$ for $t\leq t_0$ and $f^*(t)\geq A(t)$ for $t\geq t_0$. Let $x^*$ denote the unique point in the interval $(0,1-\varphi_1^{-1}(\rho_0))$ fulfilling $f^*(x^*)=1-x^*$. 
The following observation is key for the rest of the proof: For every $x\in[0,x^*]$, setting 
$y'=h_{\rho_0}^{-1}(x) \in [\varphi_1^{-1}(\rho_0),1]$, defining $f:[x,1] \rightarrow [\varphi_1^{-1}(\rho_0),1]$ by $f(t)=\frac{s_0-1+x}{t_0-x}(t-t_0)+s_0$, and letting $y''$ denote the unique point 
in $[\varphi_1^{-1}(\rho_0),1]$ fulfilling $f(y'')=y''$     
\begin{align}\label{y''y'}
y''-y'\geq\frac{r}{\sqrt{2}}
\end{align}
holds.\\
\noindent \emph{Step 1:} Define $f_1:\mathbb{R} \rightarrow \mathbb{R}$ by $f_{1}(t)=\frac{s_0-1}{t_0}(t-t_0)+s_0$ (green line with negative slope in Figure \ref{fig:construction}). Convexity of $A$ yields $ A(t)\leq f_{1}(t)$ for $t\leq t_0$ and $A(t)\geq f_{1}(t)$ for $t\geq t_0$. Let $y_{1} \in [\varphi_1^{-1}(\rho_0),1]$
denote the unique point fulfilling $f_{1}(y_{1})=y_1$ and set $x_1:=h_{\rho_{0}}(y_1)$. Then $P_{x_1,y_1}\in\mathcal{A}^\rho_{\rho_{0}}$ and $f_{1}(t)\geq P_{x_1,y_1}(t)$ for $t\leq y_1$. Case 1: If 
$x_{1} \geq  x_{*}$, then $A(t) \succ P_{x_1,y_1}(t)$, so $\rho(C_A)<\rho(C_{P_{x_1,y_1}})=\rho_{0}$, contradiction. \\
Case 2: If $x_{1} < x_{*}$ and $A(t)\geq P_{x_1,y_1}(t)$ for all $t \in (0,t_0)$, then 
$A(t) \succ P_{x_1,y_1}(t)$, contradiction. \\
Case 3: If $x_1 < x^* $ and that $A(t)< P_{x_1,y_1}(t)$ holds for some $t<t_0$ we proceed with Step 2. \\
\emph{Step 2:} Define the function $f_2: \mathbb{R} \rightarrow \mathbb{R}$ by $f_{2}(t)=\frac{s_0-1+x_1}{t_0-x_1}(t-t_0)+s_0$ (blue line starting from $(x_1,1-x_1)$ in Figure \ref{fig:construction}). Then $f_{2}(t)\geq P_{x_1,y_1}(t)$ for $t\in[x_1,y_1]$ and convexity of $A$ implies $A(t)\geq f_{2}(t)$ for $t\geq t_0$.
Let $y_2$ denote the unique point with $f_{2}(y_2)=y_2$ and set $x_2=h_{\rho_{0}}(y_2)$. 
Considering $y_2 - y_1 > \frac{r}{\sqrt{2}}$ and using ineq. (\ref{eq:lip.invh1}) we get $x_2-x_1 > \frac{r \, h_{\rho_0}'(1)}{\sqrt{2}} >0$. In case of $x_2 \geq x^*$ jump to the final step, if not, continue in the same manner to construct $x_3,x_4, \ldots, x_i$, where $i$ denote the first integer fulfilling $x_{i+1} \geq x^*$ and $x_i < x^*$ 
(notice that $x^*$ can be reached in finally many steps since 
$x_{j+1}-x_j > \frac{r \, h_{\rho_0}'(1)}{\sqrt{2}}$ holds, Figure \ref{fig:construction} depicts the case $i=4$).  \\
\emph{Final step:} Since $A \geq P_{x_i,y_i}$ implies $A \succ P_{x_i,y_i}$, which directly yields a contradiction, assume that there exists some $t \in (0,t_0)$ with $A(t) <P_{x_i,y_i}(t)$ and set
$y^*=h_{\rho_0}^{-1}(x^*) \in [y_i,y_{i+1}]$.  
Convexity of $A$ implies $A(t)\geq f_{i+1}(t) \geq P_{x*,y*}(t)$ for $t\in [t_0,1]$ as well as 
$A(t) \geq f^*(t) \geq P_{x^*,y^*}(t)$ for $t \in [0,t_0]$. Considering $P_{x^*,y^*}(t_0) < A(t_0)$
we get $A \succ P_{x^*,y^*}$ and $\rho(C_A)<  \rho(C_{P_{x^*,y^*}})=\rho_0$, contradiction.   
\end{proof}

\begin{figure}[h!]
	\begin{center}
		\includegraphics[width = 15cm]{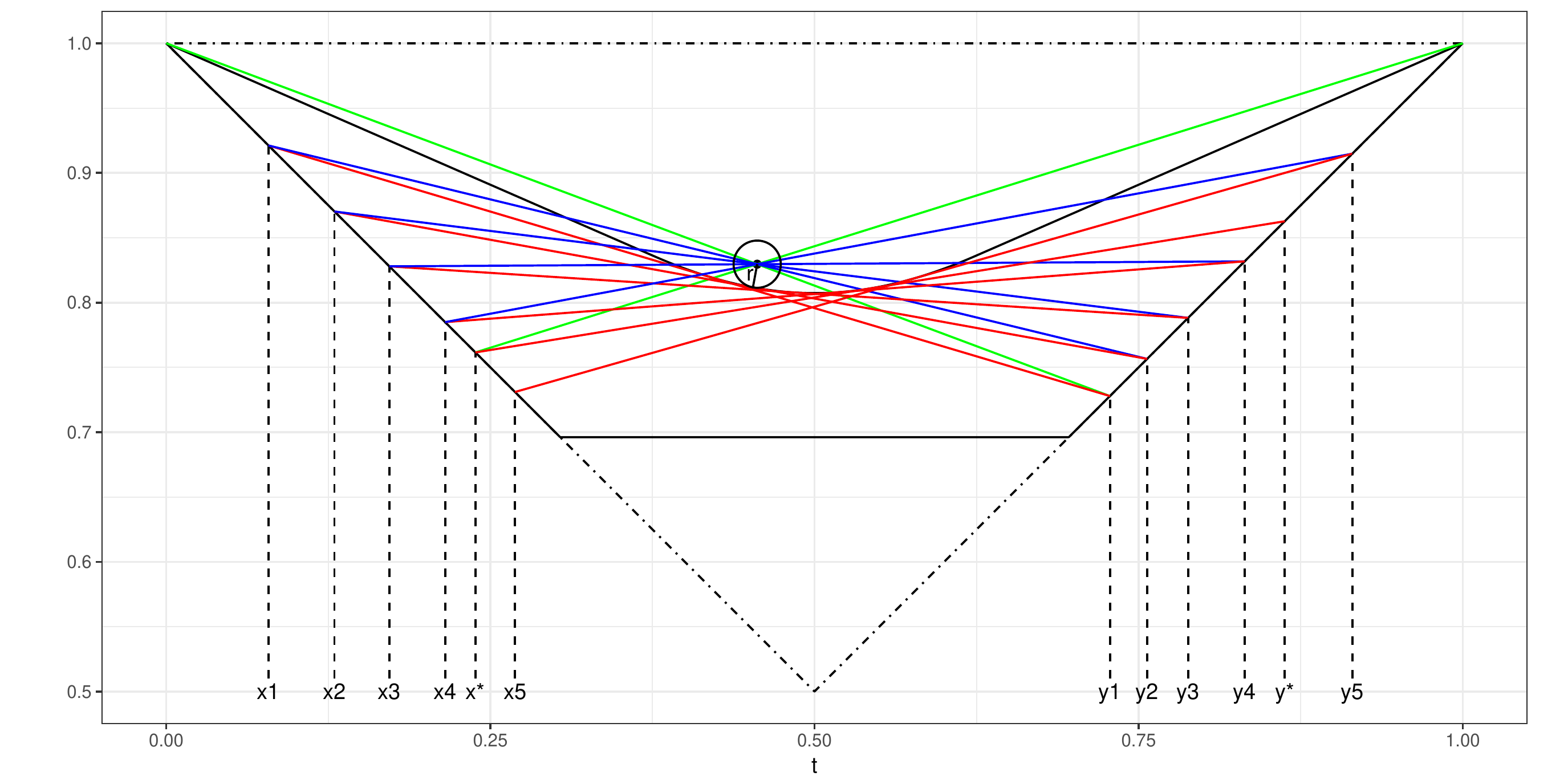}
		\caption{The construction used in the proof of Lemma \ref{upperboundrho}.}\label{fig:construction}
	\end{center}
\end{figure}

\begin{figure}[h!]
	\begin{center}
		\includegraphics[width = 11cm,angle=270]{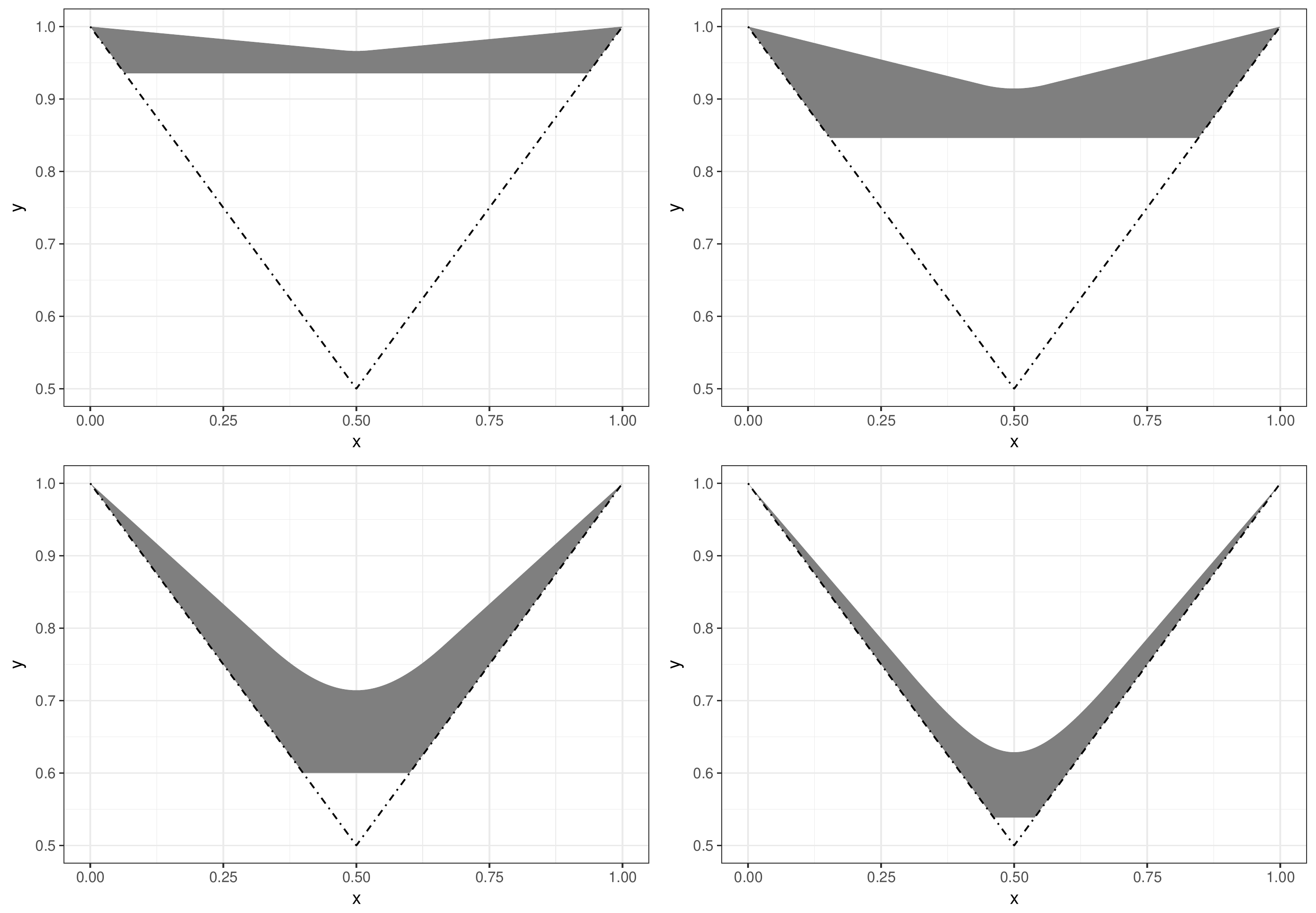}
		\caption{All $A \in \mathcal{A}^\rho_{\rho_0}$ lie in the shaded region $\Omega^\rho_{\rho_0}$. 
		  The graphic depicts the cases $\rho_{0}=0.1$ (upper left panel),  
		    $\rho_{0}=0.25$ (upper right panel), 
			for $\rho_{0}=0.75$ (lower left panel), and $\rho_{0}=0.9$ (lower right panel).}\label{fig:rho.full}
	\end{center}
\end{figure} 

\noindent Summing up, we have proved the following result, where $\Omega^\rho_{\rho_0}$ is defined by
$$
\Omega^\rho_{\rho_0}:=\big\{(t,y) \in [0,1]\times [\tfrac{1}{2},1]: L_{\varphi_1^{-1}(\rho_0)}(t) \leq y \leq U_{\rho_0}(t) \big\}.
$$
\begin{thm}\label{thm:main.rho}
Let $\rho_0 \in [0,1]$ be arbitrary but fixed. Then every $A \in \mathcal{A}^\rho_{\rho_0}$ fulfills $L_{\varphi_1^{-1}(\rho_0)} \leq  A \leq U_{\rho_0}$.
Additionally, $\Omega^\rho_{\rho_0}$ is best-possible in the sense that for each point $(t,y) \in \Omega^\rho_{\rho_0}$ there exists some $A \in \mathcal{A}^\rho_{\rho_0}$ with $A(t)=y$.
\end{thm}
\begin{proof}
The first assertion has already been proved, the second one is a direct consequence of the construction via the functions $T_{x,y}$ and $P_{x,y}$. 
\end{proof}

\noindent Figure \ref{fig:rho.full} depicts the set $\Omega^\rho_{\rho_0}$ for some choices of $\rho_0 \in [0,1]$.


\clearpage
\section{Kendall $\tau$}

It is well-known that $A \geq B$ implies $\tau(C_A) \leq \tau(C_B)$. In the sequel, however, we need strict inequality as stated in the following lemma:

\begin{lem}\label{preserveordertau}
For $A,B \in \mathcal{A}$ with $A \succ B $ we have $\tau(C_A) < \tau(C_B)$.
\end{lem}
\begin{proof}
If $A \geq B$ then $C_{A}\leq C_{B}$ holds and we get  
\begin{align}\label{eq:tau-strict}
\tau(C_{B})-\tau(C_{A})&=-1+4\int_{[0,1]^2}C_{B}d\mu_{C_{B}}+1-4\int_{[0,1]^2}C_{A}d\mu_{C_{A}} \nonumber \\
&=4\left(\int_{[0,1]^2}(C_{B}-C_{A})d\mu_{C_{B}}+\int_{[0,1]^2}C_{A}d\mu_{C_{B}}-\int_{[0,1]^2}C_{A}d\mu_{C_{A}}\right) \nonumber \\
&=4\left(\int_{[0,1]^2}(C_{B}-C_{A})d\mu_{C_{B}}+\int_{[0,1]^2}C_{B}d\mu_{C_{A}}-\int_{[0,1]^2}C_{A}d\mu_{C_{A}}\right)  \nonumber \\
&=4\left(\int_{[0,1]^2}(C_{B}-C_{A})d\mu_{C_{B}}+\int_{[0,1]^2}(C_{B}-C_{A})d\mu_{C_{A}}\right)\geq 0.
\end{align}
According to \cite{TSFS}, setting 
$$
L_D := \max\{x \in [0,1]: D(x)=1-x\}, \quad R_D := \min\{x \in [0,1]: D(x)=x\}
$$
for every $D \in \mathcal{A}$ the support $\textrm{Supp}(\mu_{C_D})$ of $\mu_{C_D}$ coincides with the set $\{(x,y) \in [0,1]^2: f_{L_D}(x) \leq y \leq f_{R_D}(x)\}$, whereby
$f_t:[0,1] \rightarrow [0,1]$ is defined as $f_t(x)=x^{\frac{1}{t}-1}$ for $t \in (0,1)$, and as
$f_0=0,f_1=1$ for $t \in \{0,1\}$. \\
Suppose now that $A \succ B $ holds and that $B$ does not coincide with $t \mapsto \max\{1-t,t\}$ (in which case $\tau(C_A)<1=\tau(C_B)$ is trivial). Obviously $L_A \leq L_B < \frac{1}{2}$ as well as $R_A \geq R_B > \frac{1}{2}$ and we can find some $t_0 \in (L_B,R_B)$ fulfilling
$A(t_0)>B(t_0)$. By continuity there exists some $\delta>0$ such that $A(t)>B(t)$ holds for every
$t \in [\underline{t},\overline{t}]:=[t_0 - \delta, t_0 + \delta] \subseteq (L_B,R_B) \subseteq (L_A,R_A)$. Considering 
$$
\big\{(x,y) \in (0,1)^2: f_{\underline{t}}(x) < y < f_{\overline{t}}(x) \big\} \subset 
\big\{ (x,y) \in (0,1)^2: C_B(x,y) > C_A(x,y)\big\}=:\{C_B > C_A\}
$$
and
$$
\big\{(x,y) \in (0,1)^2: f_{\underline{t}}(x) < y < f_{\overline{t}}(x) \big\} \subset 
\textrm{Supp}(\mu_{C_B}) \subseteq \textrm{Supp}(\mu_{C_A})
$$
shows $\mu_{C_B}(\{C_B > C_A\})>0$ and $\mu_{C_A}(\{C_B > C_A\})>0$. Hence 
$$
\int_{[0,1]^2}(C_{B}-C_{A})d \mu_{C_{B}} >0 \, \textrm{ and }
\int_{[0,1]^2}(C_{B}-C_{A})d\mu_{C_{A}} > 0
$$
follows, and applying eq. (\ref{eq:tau-strict}) yields $\tau(C_A) < \tau(C_B)$.
\end{proof}

\noindent Defining $\varphi_2:[\frac{1}{2},1] \rightarrow [0,1]$ by 
$\varphi_2(y)=-1+\frac{1}{y}$
we get 
\begin{equation}
\tau(C_{T_{x,y}})=\varphi_2(y)
\end{equation}
for every $T_{x,y}$ with $x\in[1-y,y]$, i.e. $\tau(C_{T_{x,y}})$ does not depend on $x\in[1-y,y]$. Obviously $\varphi_2$ is a strictly decreasing homeomorphism mapping $[\frac{1}{2},1]$ onto 
$[0,1]$. Furthermore, letting $\varphi_2^{-1}(\tau)=\frac{1}{1+\tau}$ denote its inverse,
\begin{equation}
T_{x,\varphi_2^{-1}(\tau_0)}\in\mathcal{A}^\tau_{\tau_0}
\end{equation}
holds for every $x\in[1-\varphi_2^{-1}(\tau_0),\varphi_2^{-1}(\tau_0)]$.\\
\noindent A straightforward calculation yields 
\begin{equation}
\tau(C_{P_{x,y}})=\frac{1-3x-y+4xy}{y-x}.
\end{equation}
for $y>x$ and $\tau(P_{x,y})=\tau(P_{y,y})=\tau(M)$ for $x=y=\frac{1}{2}$. 
Considering $L_{y}=P_{1-y,y}$ and defining $\psi_2:[\frac{1}{2},1]\to[0,1]$ by 
$$\psi_2(y)=2(1-y)$$
we get 
\begin{equation}
\tau(C_{L_{y}})=\psi_2(y)
\end{equation}
for every $y \in [\frac{1}{2},1]$. Obviously $\psi_2$ is a strictly decreasing homeomorphism mapping $[\frac{1}{2},1]$ onto $[0,1]$. Letting $\psi_2^{-1}(\tau)=\frac{2-\tau}{2}$ denote its inverse, 
\begin{equation*}
L_{\psi_2^{-1}(\tau_0)} \in \mathcal{A}^\tau_{\tau_0}
\end{equation*}
holds for every $\tau_0\in[0,1]$.  

For every $y \in [\frac{1}{2},1]$ and every $x \in [0,\frac{1}{2}]$ obviously 
$\tau(C_{P_{x,y}}) \in [\tau(C_{T_{y,y}}),1]=[\varphi_2(y),1]$. Additionally, for fixed $y \in (\frac{1}{2},1]$ and $0 \leq x_1 < x_2 \leq 1$ we have 
$P_{x_1,y} \succ P_{x_2,y}$, implying that the mapping $x \mapsto \tau(C_{P_{x,y}})$ is an increasing homeomorphism mapping 
$[0,\frac{1}{2}]$ onto $[\varphi_2(y),1]$. 

Again consider $\tau_0 \in [0,1)$. Then for every $y \geq \varphi_2^{-1}(\tau_0) > \frac{1}{2}$ there exists exactly one $x \in [0,\frac{1}{2})$ with $\tau(C_{P_{x,y}})=\tau_0$. It is easy to verify that this $x$ is given by 
$$
x=\frac{-1+y+\tau_0y}{-3+\tau_0+4y}:=h_{\tau_0}(y)
$$
and that $h_{\tau_{0}}(\varphi_2^{-1}(\tau_0))=0$ as well as 
$h_{\tau_{0}}(1)= \frac{\tau_0}{1+\tau_0}= 1- \varphi_2^{-1}(\tau_0)$ holds.
Since the derivative $h_{\tau_{0}}'$ of $h_{\tau_{0}}$ is given by
$$
h'_{\tau_{0}}(y)=\frac{(1-\tau_0)^2}{(-3+\tau_0+4y)^2} 
$$ 
we get $\min_{y\in[\varphi_2^{-1}(\tau_{0}),1]}h_{\tau_{0}}'(y)=h_{\tau_{0}}'(1)=\left(\frac{1-\tau_0}{1+\tau_0}\right)^2>0$. 
As direct consequence, $h_{\rho_{0}}$ is strictly increasing and for $\varphi_2^{-1}(\tau_0) \leq \underline{y} < \overline{y} \leq 1$ 
\begin{align}\label{eq:lip.invh2}
h_{\rho_{0}}(\overline{y})-h_{\rho_{0}}(\underline{y}) \geq h_{\rho_{0}}'(1)(\overline{y}-\underline{y}) = \left(\frac{1-\tau_0}{1+\tau_0}\right)^2(\overline{y}-\underline{y})
\end{align} 
holds. This non-contractivity property will be key in the proof of Lemma \ref{upperboundtau}.

For every $y\in[\varphi_2^{-1}(\tau_{0}),1]$, $P_{h_{\tau_0}(y),y}(\frac{1}{2})=1-\frac{\tau_0}{2}=\psi_2^{-1}(\tau_0)$, so the upper envelope $U_{\tau_0}$
of the family $(P_{h_{\tau_0}(y),y})_{y \in [\varphi_2^{-1}(\tau_0),1]}$ is given by 
\begin{align}\label{eq:Utau}
U_{\tau_0}(t)&=
\begin{cases}
P_{h_{\tau_0}(\varphi_2^{-1}(\tau_{0})),\varphi_2^{-1}(\tau_{0})}(t) = P_{0,\frac{1}{1+\tau_0}}(t)& \text{if }t\in[0,\frac{1}{2}] \\
P_{h_{\tau_0}(1),1}(t)=P_{\frac{\tau_0}{1+\tau_0},1}(t) & \text{if }t\in(\frac{1}{2},1] \\
\end{cases} \notag \\
&=T_{\frac{1}{2},\psi_2^{-1}(\tau_0)}(t)=T_{\frac{1}{2},1-\frac{\tau_0}{2}}(t)
\end{align}
The function $U_{\tau_0}$ is symmetric w.r.t. $t=\frac{1}{2}$ and convex on $[0,1]$. 

We now show that $L_{\varphi_2^{-1}(\tau_0)} \leq A \leq U_{\tau_0}$ holds for every $A \in \mathcal{A}^\tau_{\rho_0}$ and $\rho_0 \in [0,1]$ and start with the lower bound.

\begin{lem}
	Suppose that $\tau_0 \in [0,1]$. Then every $A \in \mathcal{A}^\tau_{\tau_0}$ fulfills $A \geq L_{\varphi_2^{-1}(\tau_0)}$.
\end{lem} 

\begin{proof}
	The assertion is trivial for $\tau_0 \in \{0,1\}$, so we may consider $\tau_0 \in (0,1)$. Suppose that $A \in \mathcal{A}^\tau_{\tau_0}$ and 
	that there exists $t_0 \in (1-\varphi_2^{-1}(\tau_0),\varphi_2^{-1}(\tau_0))$ such that $A(t_0) < L_{\varphi_2^{-1}(\tau_0)}(t_0)=\varphi_2^{-1}(\tau_0)$ holds. 
	Convexity of $A$ together with $A(0)=A(1)=1$ implies $A \leq T_{t_0,A(t_0)}$, from which we immediately get 
	$$
	\tau(C_A) \geq \tau(C_{T_{t_0,A(t_0)}}) = \varphi_2 \circ A(t_0) > \varphi_2 \circ \varphi_2^{-1}(\tau_0)=\tau_0,
	$$ 
	a contradiction to $A \in \mathcal{A}^\tau_{\tau_0}$.	
\end{proof}	

\begin{lem}\label{upperboundtau}
Suppose that $\tau_0 \in [0,1]$. Then every $A \in \mathcal{A}^\tau_{\tau_0}$ fulfills $A \leq U_{\tau_0}$.
\end{lem} 

\begin{proof}
For the extreme cases $\tau_0=0$ and $\tau_0=1$ we get $U_{0}=1$ and $U_1=T_{\frac{1}{2},\frac{1}{2}}$ respectively, so the result obviously holds. 
For the rest of the proof we consider $\tau_0 \in (0,1)$.  
Suppose that $A\in\mathcal{A}^\tau_{\tau_{0}}$ and that $s_0:=A(t_{0})>U_{\tau_{0}}(t_{0})$ for some $t_{0}\in[0,1]$. The case $s_0=1$ yields a contradiction immediately, so we assume $s_0 < 1$. Symmetry of $U_{\tau_0}$ implies that it suffices to consider $t_0 \in (0,\frac{1}{2}]$, continuity of $U_{\tau_0}$ yields $r=\min_{t\in[0,1]}\sqrt{(t-t_0)^2+(U_{\tau_{0}}(t)-s_0)^2}>0$.\\
\noindent Define $f^*$ by $f^*(t)=\frac{s_0-1}{t_0-1}(t-t_0)+s_0$ (increasing green with positive slope in Figure \ref{fig:constructiontau}). Convexity of $A$ 
yields $f^*(t)\leq A(t)$ for $t\leq t_0$ and $f^*(t)\geq A(t)$ for $t\geq t_0$. Let $x^*$ denote the unique point in the interval $(0,1-\varphi_2^{-1}(\tau_0))$ fulfilling $f^*(x^*)=1-x^*$. 
The following observation is key for the rest of the proof: For every $x\in[0,x^*]$, setting 
$y'=h_{\tau_0}^{-1}(x) \in [\varphi_2^{-1}(\tau_0),1]$, defining $f:[x,1] \rightarrow [\varphi_2^{-1}(\tau_0),1]$ by $f(t)=\frac{s_0-1+x}{t_0-x}(t-t_0)+s_0$, and letting $y''$ denote the unique point 
in $[\varphi_2^{-1}(\tau_0),1]$ fulfilling $f(y'')=y''$     
\begin{align}\label{y''y'tau}
y''-y'\geq\frac{r}{\sqrt{2}}
\end{align}
holds.\\
\noindent \emph{Step 1:} Define $f_1:\mathbb{R} \rightarrow \mathbb{R}$ by $f_{1}(t)=\frac{s_0-1}{t_0}(t-t_0)+s_0$ (green line with negative slope in Figure \ref{fig:constructiontau}). Convexity of $A$ yields $ A(t)\leq f_{1}(t)$ for $t\leq t_0$ and $A(t)\geq f_{1}(t)$ for $t\geq t_0$. Let $y_{1} \in [\varphi_2^{-1}(\tau_0),1]$
denote the unique point fulfilling $f_{1}(y_{1})=y_1$ and set $x_1:=h_{\tau_{0}}(y_1)$. Then $P_{x_1,y_1}\in\mathcal{A}^\tau_{\tau_{0}}$ and $f_{1}(t)\geq P_{x_1,y_1}(t)$ for $t\leq y_1$. Case 1: If 
$x_{1} \geq  x_{*}$, then $A(t) \succ P_{x_1,y_1}(t)$, so $\tau(C_A)<\tau(C_{P_{x_1,y_1}})=\tau_{0}$, contradiction. \\
Case 2: If $x_{1} < x^*$ and $A(t)\geq P_{x_1,y_1}(t)$ for all $t \in (0,t_0)$, then 
$A(t) \succ P_{x_1,y_1}(t)$, contradiction. \\
Case 3: If $x_1 < x^* $ and that $A(t)< P_{x_1,y_1}(t)$ holds for some $t<t_0$ we proceed with Step 2. \\
\emph{Step 2:} Define the function $f_2: \mathbb{R} \rightarrow \mathbb{R}$ by $f_{2}(t)=\frac{s_0-1+x_1}{t_0-x_1}(t-t_0)+s_0$ (blue line starting from $(x_1,1-x_1)$ in Figure \ref{fig:constructiontau}). Then $f_{2}(t)\geq P_{x_1,y_1}(t)$ for $t\in[x_1,y_1]$ and convexity of $A$ implies $A(t)\geq f_{2}(t)$ for $t\geq t_0$.
Let $y_2$ denote the unique point with $f_{2}(y_2)=y_2$ and set $x_2=h_{\tau_{0}}(y_2)$. 
Considering $y_2 - y_1 > \frac{r}{\sqrt{2}}$ and using ineq. (\ref{eq:lip.invh1}) we get $x_2-x_1 > \frac{r \, h_{\tau_0}'(1)}{\sqrt{2}} >0$. In case of $x_2 \geq x^*$ jump to the final step, if not, continue in the same manner to construct $x_3,x_4, \ldots, x_i$, where $i$ denote the first integer fulfilling $x_{i+1} \geq x^*$ and $x_i < x^*$ 
(notice that $x^*$ can be reached in finally many steps since 
$x_{j+1}-x_j > \frac{r \, h_{\tau_0}'(1)}{\sqrt{2}}$ holds, Figure \ref{fig:constructiontau} depicts the case $i=3$).  \\
\emph{Final step:} Since $A \geq P_{x_i,y_i}$ implies $A \succ P_{x_i,y_i}$, which directly yields a contradiction, assume that there exists some $t \in (0,t_0)$ with $A(t) <P_{x_i,y_i}(t)$ and set
$y^*=h_{\tau_0}^{-1}(x^*) \in [y_i,y_{i+1}]$.  
Convexity of $A$ implies $A(t)\geq f_{i+1}(t) \geq P_{x^*,y^*}(t)$ for $t\in [t_0,1]$ as well as 
$A(t) \geq f^*(t) \geq P_{x^*,y^*}(t)$ for $t \in [0,t_0]$. Considering $P_{x^*,y^*}(t_0) < A(t_0)$
we get $A \succ P_{x^*,y^*}$ and $\tau(C_A)<  \tau(C_{P_{x^*,y^*}})=\tau_0$, contradiction.   
\end{proof}

\begin{figure}[h!]
	\begin{center}
		\includegraphics[width = 15cm]{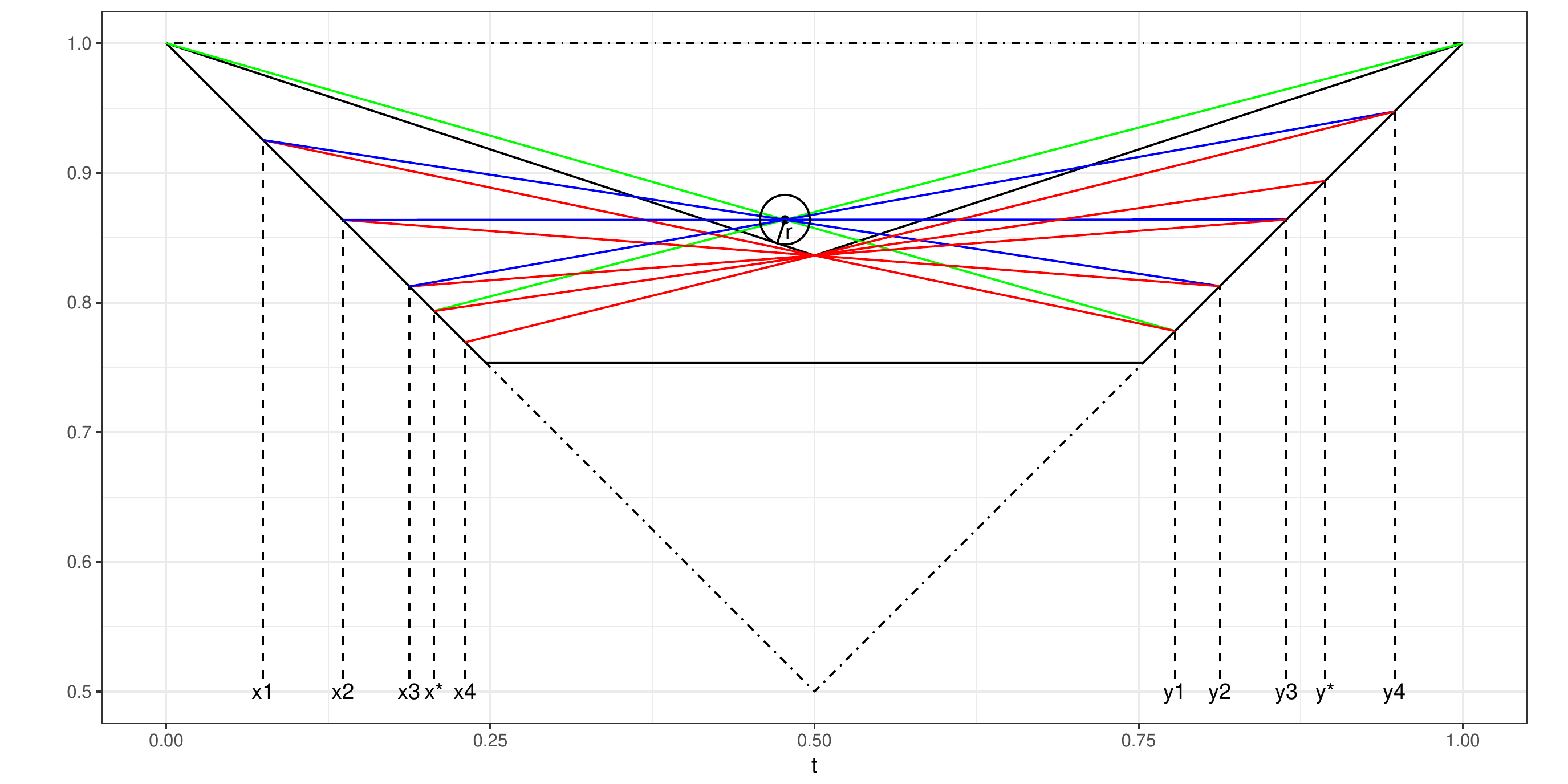}
		\caption{The construction used in the proof of Lemma \ref{upperboundtau}.}\label{fig:constructiontau}
	\end{center}
\end{figure}

\begin{figure}[h!]
	\begin{center}
		\includegraphics[width = 11cm,angle=270]{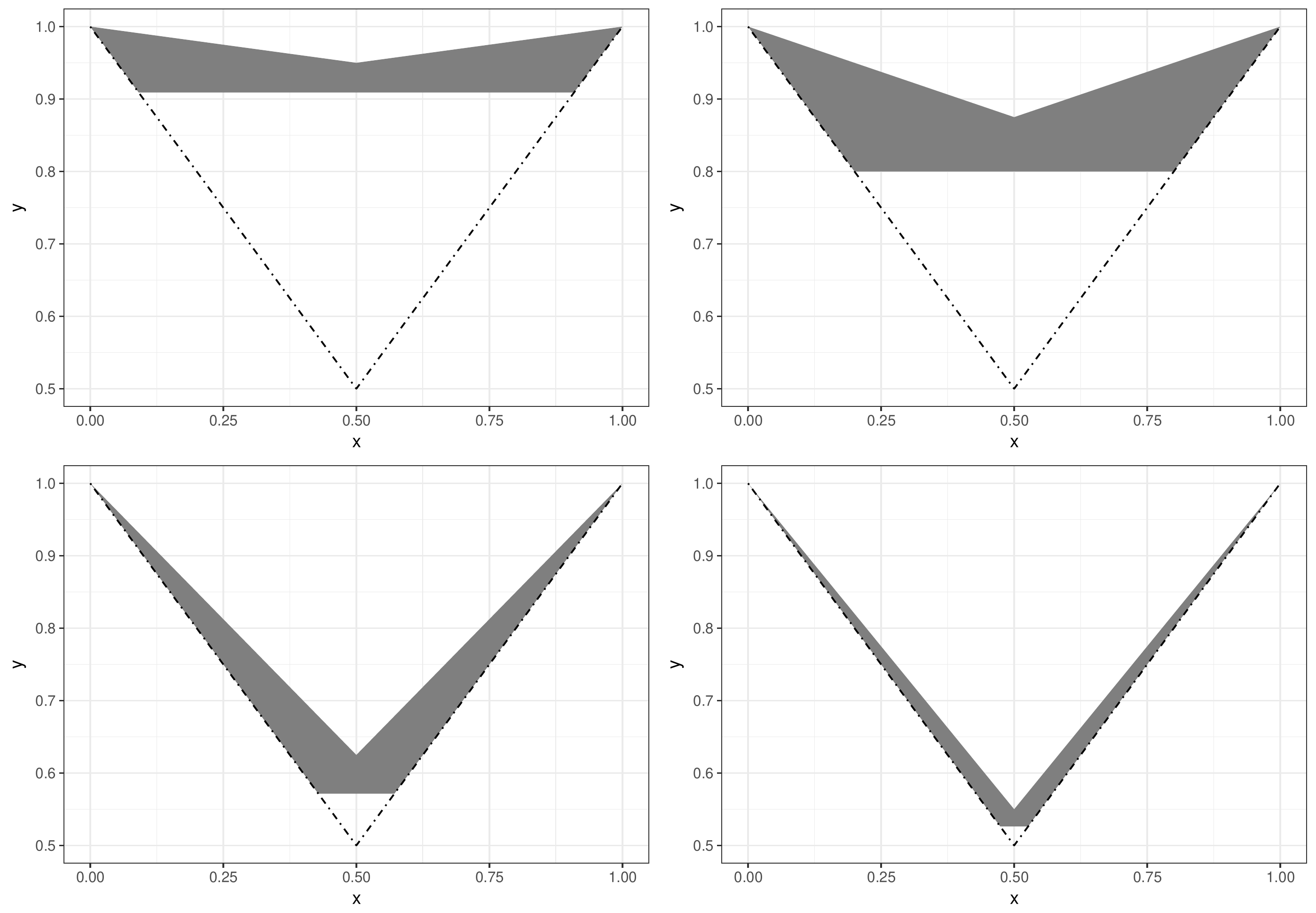}
		\caption{All $A \in \mathcal{A}^\tau_{\tau_0}$ lie in the shaded region $\Omega^\tau_{\tau_0}$. 
			The graphic depicts the cases $\tau_{0}=0.1$ (upper left panel),  
			$\tau_{0}=0.25$ (upper right panel), 
			for $\tau_{0}=0.75$ (lower left panel), and $\tau_{0}=0.9$ (lower right panel).}\label{fig:tau.full}
	\end{center}
\end{figure} 

\noindent Summing up, we have proved the following result, where $\Omega^\tau_{\tau_0}$ is defined by
$$
\Omega^\tau_{\tau_0}:=\big\{(t,y)\in [0,1]\times [\tfrac{1}{2},1]: L_{\varphi_2^{-1}(\tau_0)}(t) \leq y \leq U_{\tau_0}(t) \big\}.
$$
\begin{thm}
	Let $\tau_0 \in [0,1]$ be arbitrary but fixed. Then every $A \in \mathcal{A}^\tau_{\tau_0}$ fulfills $L_{\varphi_2^{-1}(\tau_0)} \leq  A \leq U_{\tau_0}$.
	Additionally, $\Omega^\tau_{\tau_0}$ is best-possible in the sense that for each point $(t,y) \in \Omega^\tau_{\tau_0}$ there exists some $A \in \mathcal{A}^\tau_{\tau_0}$ with $A(t)=y$.
\end{thm}
\begin{proof}
	The first assertion has already been proved, the second one is a direct consequence of the construction via the functions $T_{x,y}$ and $P_{x,y}$. 
\end{proof}
\noindent Figure \ref{fig:tau.full} depicts the set $\Omega^\tau_{\tau_0}$ for some choices of $\tau_0 \in [0,1]$.



\clearpage
\section{Appendix}
\subsection{Complementary calculations for Section \ref{sec:rho.full}.}
Calculation of the upper envelope of the family $(P_{h_{\rho_{0}(y)},y})_{y \in [\varphi_1^{-1}(\rho_0),1]}$ with $\rho_0 \in (0,1)$: 
Define 
\begin{align*}
g_{\rho_0,t}(y)&:=P_{h_{\rho_{0}(y)},y}(t) \\
&=\frac{3t+\rho_0t+3y-3\rho_0y-18ty+2\rho_0ty+3y^2-3\rho_0y^2+15ty^2+\rho_0ty^2}{6-2\rho_0-15y-\rho_0y+15y^2+\rho_0y^2}
\end{align*}
for every $t\in(h_{\rho_{0}}(y),y)$. Then 
\begin{align*}
g_{\rho_0,t}'(y)&=\frac{3 - 3 \rho_0 - 18 t + 2\rho_0 t + 6 y - 6 \rho_0 y + 30 t y + 2 \rho_0 t y}{6 - 2 \rho_0 - 15 y - \rho_0 y + 15 y^2 + \rho_0 y^2} \\
& \vspace{-2mm} -\frac{(-15 -\rho_0 + 30 y + 2\rho_0 y) (3 t + \rho_0 t + 3 y - 3 \rho_0 y - 18 t y + 2 \rho_0 t y + 3 y^2 - 3 \rho_0 y^2 + 15 t y^2 + \rho_0 t y^2)}{(6 - 2 \rho_0 - 15 y - \rho_0 y + 15 y^2 + \rho_0 y^2)^2}.
\end{align*}
The solutions of $g_{\rho_0,t}'(y)=0$ are
$$\frac{-6 + 2 \rho_0 - 15 t -\rho_0 t \pm 6\sqrt{ 6-2\rho_0 - 15t(1-t) - \rho_{0}t(1-t)}}{-30 - 2\rho_0 + 15 t + \rho_0 t}.$$
Since $\frac{-6 + 2 \rho_0 - 15 t -\rho_0 t + 6\sqrt{6-2\rho_0 - 15t(1-t) - \rho_{0}t(1-t)}}{-30 - 2\rho_0 + 15 t + \rho_0 t}<\frac{1}{2}$ for all $t\in[0,1]$ choose 
\begin{align}
y_0=\frac{-6 + 2 \rho_0 - 15 t -\rho_0 t - 6\sqrt{6-2\rho_0 - 15t(1-t) - \rho_{0}t(1-t)}}{-30 - 2\rho_0 + 15 t + \rho_0 t}=:y_{\rho_{0}}(t)
\end{align}

\begin{lem}\label{equipoint}
$y_0\in[\varphi_1^{-1}(\rho_0),1]$ if and only if $t\in[\frac{3-\rho_{0}}{6+\rho_{0}},\frac{3+2\rho_{0}}{6+\rho_{0}}]$.
\end{lem}

\begin{proof}
($\Leftarrow$) We have
$$y_{\rho_{0}}'(t)=\frac{9\left(-6+\rho_{0}(t-2)+15t+4\sqrt{6-2\rho_0 - 15t(1-t) - \rho_{0}t(1-t)}\right)}{(15+\rho_{0})(t-2)^2\sqrt{6-2\rho_0 - 15t(1-t) - \rho_{0}t(1-t)}}$$
and $6-2\rho_0 - 15t(1-t) - \rho_{0}t(1-t)>0$ for all $\rho_{0}<1$. A straightforward calculation shows $y_{\rho_{0}}(\frac{3-\rho_{0}}{6+\rho_{0}})=\frac{3-\rho_{0}}{3+\rho_{0}}=\varphi_1^{-1}(\rho_{0})$ and $y_{\rho_{0}}(\frac{3+2\rho_{0}}{6+\rho_{0}})=1$ and it suffices to show that  $y_{\rho_{0}}'(t)>0$ for all $t\in(\frac{3-\rho_{0}}{6+\rho_{0}},\frac{3+2\rho_{0}}{6+\rho_{0}})$. 
The condition $y_{\rho_{0}}'(t)=0$ is equivalent to  
\begin{align*}
0&=-6+\rho_{0}(t-2)+15t+4\sqrt{6-2\rho_0 - 15t(1-t) - \rho_{0}t(1-t)} \\
\frac{1}{4}\left(6-\rho_{0}(t-2)-15t\right)&=\sqrt{6-2\rho_0 - 15t(1-t) - \rho_{0}t(1-t)} \\
\frac{1}{16}\left(6-\rho_{0}(t-2)-15t\right)^2&=6-2\rho_0 - 15t(1-t) - \rho_{0}t(1-t) \\
0&=(\rho_{0}-1)(15+\rho_{0})(t-2)^2.
\end{align*}
Of the latter, $t=2\notin(\frac{3-\rho_{0}}{6+\rho_{0}},\frac{3+2\rho_{0}}{6+\rho_{0}})$ is the only solution, so $y_{\rho_{0}}'(t)\neq0$. Obviously $y_{\rho_{0}}'$ is continuous on $(0,1)$ and $y_{\rho_{0}}'(\frac{1}{2})=\frac{4(1-\rho_{0}+4\sqrt{1-\rho_{0}})}{(15+\rho_{0})\sqrt{1-\rho_{0}}}>0$. Hence $y_{\rho_{0}}'(t)>0$ for all $t\in(\frac{3-\rho_{0}}{6+\rho_{0}},\frac{3+2\rho_{0}}{6+\rho_{0}})$ and the result follows.

($\Rightarrow$) Solving $y_0=y_{\rho_{0}}(t)$ for $t$ we get $t=\frac{2(-3+\rho_{0}-6y_0+2\rho_{0}y_0+15y_0^2+\rho_{0}y_0^2)}{-21+\rho_{0}+30y_0+2\rho_{0}y_0+15y_0^2+\rho_{0}y_0^2}=:t_{\rho_{0}}(y_0)$. 
Straightforward calculations yield $t_{\rho_{0}}(\varphi_1^{-1}(\rho_{0}))=\frac{3-\rho_{0}}{6+\rho_{0}}$ and $t_{\rho_{0}}(1)=\frac{3+2\rho_{0}}{6+\rho_{0}}$. Moreover, $t_{\rho_{0}}$ is not continuous at $y_0=\frac{-15-\rho_{0}\pm6\sqrt{15+\rho_{0}}}{15+\rho_{0}}$. Since $y_0=\frac{-15-\rho_{0}-6\sqrt{15+\rho_{0}}}{15+\rho_{0}}<0$ and 
$\varphi_1^{-1}(\rho_{0})-\frac{-15-\rho_{0}+6\sqrt{15+\rho_{0}}}{15+\rho_{0}}=\frac{6}{3+\rho_{0}}-\frac{6}{\sqrt{15+\rho_{0}}}>0$ for all $\rho_{0}<1$, $t_{\rho_{0}}$ is continuous on $[\varphi_1^{-1}(\rho_{0}),1]$. 
Considering that 
$$
t_{\rho_{0}}'(y_0)=\frac{72(\rho_{0}(-2-y_0+y_0^2)+3(2-5y_0+y_0^2))}{(-21+\rho_{0}+30y_0+2\rho_{0}y_0+15y_0^2+\rho_{0}y_0^2)^2}>0
$$
holds for all $y_0\in(\varphi_1^{-1}(\rho_{0}),1)$, $t_{\rho_{0}}$ is increasing on $[\varphi_1^{-1}(\rho_{0}),1]$ and $t=t_{\rho_{0}}(y_0)\in[\frac{3-\rho_{0}}{6+\rho_{0}},\frac{3+2\rho_{0}}{6+\rho_{0}}]$.
\end{proof}

\noindent Define the function $S_{\rho_{0}}$ on the interval $[\frac{3-\rho_{0}}{6+\rho_{0}},\frac{3+2\rho_{0}}{6+\rho_{0}}]$ by
\begin{align}
S_{\rho_0}(t)&=g_{\rho_0,t}(y_0)=\frac{9-\rho_0+4\sqrt{6-2\rho_0 - 15t(1-t) - \rho_{0}t(1-t)}}{15+\rho_0}.\label{eq:up1} 
\end{align}
Considering 
$S_{\rho_0}''(t)=\frac{9(1-\rho_{0})}{(6-2\rho_0 - 15t(1-t) - \rho_{0}t(1-t))^{\frac{3}{2}}}>0$
for all  $t\in(\frac{3-\rho_{0}}{6+\rho_{0}},\frac{3+2\rho_{0}}{6+\rho_{0}})$, $S_{\rho_0}$ is convex. For the boundary points of $[\varphi_1^{-1}(\rho_0),1]$, we have 
\begin{align}
g_{\rho_{0},t}(\varphi_1^{-1}(\rho_0))&=\frac{-3+\rho_{0}+2\rho_{0}t}{-3+\rho_{0}} \label{eq:up2}\\
g_{\rho_{0},t}(1)&=\frac{-3+3\rho_{0}-2\rho_{0}t}{-3+\rho_{0}}. \label{eq:up3}
\end{align} 
Solving \eqref{eq:up1}$=$\eqref{eq:up3} with respect to $t$ yields $t=\frac{3+2\rho_{0}}{6+\rho_{0}}$. Moreover,
\begin{align*}
g_{\rho_{0},\frac{1}{2}}(1)&=\frac{-3+2\rho_{0}}{-3+\rho_{0}} \\
S_{\rho_{0}}(\tfrac{1}{2})&=\frac{9-\rho_{0}+6\sqrt{1-\rho_{0}}}{15+\rho_{0}}.
\end{align*}
Solving $g_{\rho_{0},\frac{1}{2}}(1)=S_{\rho_{0}}(\frac{1}{2})$, we get $\rho_{0}=0$ or $\rho_{0}=1$ and $g_{\frac{3}{4},\frac{1}{2}}(1)=\frac{2}{3}<\frac{5}{7}= S_{\frac{3}{4}}(\frac{1}{2})$. So   $g_{\rho_{0},\frac{1}{2}}(1)\leq S_{\rho_{0}}(\frac{1}{2})$ for every $\rho_{0}$. Since $S_{\rho_{0}}$ is convex, $g_{\rho_0,t}(1)$ is linear function w.r.t. $t$, and since there is only one point fulfilling $g_{\rho_{0},t}(1)= S_{\rho_{0}}(t)$, $g_{\rho_{0},t}(1)\leq S_{\rho_{0}}(t)$ holds for all $t\in[\frac{1}{2},\frac{3+2\rho_{0}}{6+\rho_{0}}]$. Analogously we get $g_{\rho_{0},t}(\varphi_1^{-1}(\rho_0))\leq S_{\rho_{0}}(t)$ for all $t\in[\frac{3-\rho_{0}}{6+\rho_{0}},\frac{1}{2}]$. Altogether, $y_0$ is the global maximum of the function $g_{\rho_{0},t}$ on the interval 
$[\varphi_1^{-1}(\rho_0),1]$ for every $t\in[\frac{3-\rho_{0}}{6+\rho_{0}},\frac{3+2\rho_{0}}{6+\rho_{0}}]$.


\end{document}